\documentclass[11pt]{amsart}

\usepackage{amsmath}
\usepackage{amssymb,mathtools,amsthm,latexsym}
\usepackage{mathrsfs}
\usepackage{bm}
\usepackage{url}
\usepackage{xcolor}
\numberwithin{equation}{section}

\calclayout

\usepackage{comment}
\usepackage{soul}

\usepackage{todonotes}
\newtheorem{thm}{Theorem}[section]
\newtheorem{lemma}[thm]{Lemma}
\newtheorem{cor}[thm]{Corollary}
\theoremstyle{definition}
\newtheorem{defn}[thm]{Definition}

\newtheorem{remark}[thm]{Remark}
\newtheorem{prop}[thm]{Proposition}

\newtheorem{question}[thm]{Question}

\title[Tensorization of quasi-Hilbertian Sobolev Spaces]{Tensorization of quasi-Hilbertian Sobolev Spaces}
\author{Sylvester Eriksson-Bique}
\author{Tapio Rajala}
\author{Elefterios Soultanis}

\address{Sylvester Eriksson-Bique\\
Research Unit of Mathematical Sciences \\
P.O.Box 3000\\
FI-90014 Oulu}
\email{\tt sylvester.d.eriksson-bique@jyu.fi}

\address{Tapio Rajala\\
University of Jyvaskyla \\
Department of Mathematics and Statistics \\
P.O. Box 35 (MaD) \\
FI-40014 University of Jyvaskyla \\
Finland}
\email{\tt tapio.m.rajala@jyu.fi}

\address{Elefterios Soultanis\\
Radboud University\\
Department of Mathematics\\
P.O. Box 9010, Postvak 59 \\
6500 GL Nijmegen\\
Netherlands}
\email{\tt elefterios.soultanis@gmail.com}

\def\XXint#1#2#3{{\setbox0=\hbox{$#1{#2#3}{\int}$}
		\vcenter{\hbox{$#2#3$}}\kern-.5\wd0}}

\newcommand{\N}{\ensuremath{\mathbb{N}}}
\newcommand{\R}{\ensuremath{\mathbb{R}}}

\newcommand{\LIP}{\ensuremath{\mathrm{LIP}}}

\newcommand{\Lip}{\ensuremath{\mathrm{Lip}}}

\newcommand{\defeq}{\mathrel{\mathop:}=}

\newcommand{\Mod}{\ensuremath{\mathrm{Mod}}}

\newcommand{\eps}{\varepsilon}

\newcommand{\ud}{\mathrm{d}}

\newcommand{\inv}{^{-1}}

\newcommand{\cE}{\mathcal{E}}
\newcommand{\cD}{\mathcal{D}}

\begin{document}

\maketitle

\begin{abstract} The tensorization problem for Sobolev spaces asks for a characterization of how the Sobolev space on a product metric measure space $X\times Y$ can be determined from its factors.
We show that two natural descriptions of the Sobolev space from the literature coincide, $W^{1,2}(X\times Y)=J^{1,2}(X,Y)$, thus settling the tensorization problem for Sobolev spaces in the case $p=2$, when $X$ and $Y$ are \emph{infinitesimally quasi-Hilbertian}, i.e. the Sobolev space $W^{1,2}$ admits an equivalent renorming by a Dirichlet form. This class includes in particular metric measure spaces $X,Y$ of finite Hausdorff dimension as well as infinitesimally Hilbertian spaces.


More generally for $p\in (1,\infty)$ we obtain the norm-one inclusion $\|f\|_{J^{1,p}(X,Y)}\le \|f\|_{W^{1,p}(X\times Y)}$ and show that the norms agree on the algebraic tensor product $W^{1,p}(X)\otimes W^{1,p}(Y)\subset W^{1,p}(X\times Y)$. When $p=2$ and $X$ and $Y$ are infinitesimally quasi-Hilbertian, standard Dirichlet form theory yields the density of $W^{1,2}(X)\otimes W^{1,2}(Y)$ in $J^{1,2}(X,Y)$ thus implying the equality of the spaces. Our approach raises the question of the density of $W^{1,p}(X)\otimes W^{1,p}(Y)$ in $J^{1,p}(X,Y)$ in the general case.

\end{abstract}

\section{Introduction}

Over the last three decades Sobolev spaces over \emph{metric spaces} have become a prominent feature in a plethora of geometric problems ranging from Plateau-type problems \cite{fit-wen21, lyt17'',lyt18} to quasiconformal uniformization questions \cite{lyt-wen20, raj17} and structural problems of spaces with Ricci curvature bounds \cite{ags14a, ags14b, gig13,gig18}. During that time their theory has been studied intensively and significant developments include the unification of different definitions of Sobolev spaces, the density (in energy) of Lipschitz functions in, and the reflexivity of Sobolev spaces over metric spaces in a very general setting - see e.g. \cite{amb15,che99,teriseb,gig15,sha00}. 

The \emph{tensorization problem for Sobolev spaces}, first considered in \cite{ags14b}, asks whether Sobolev regularity of a function of two variables can be deduced from the existence and integrability of directional derivatives. More precisely, let $X=(X,d_X,\mu)$ and $Y=(Y,d_Y,\nu)$ be two metric measure spaces, $p\in [1,\infty)$ 
and $\left(X\times Y, \sqrt{d_X^2 + d_Y^2}, \mu \times \nu\right)$ their (Euclidean) product. Given $p\ge 1$, the tensorization problem asks whether the Sobolev space $W^{1,p}(X\times Y)$ coincides with the \emph{Beppo--Levi space} $J^{1,p}(X,Y)$ consisting of functions $f\in L^p(X\times Y)$ for which  $f(x,\cdot) \in W^{1,p}(Y)$ for $\mu$-almost every $x \in X$,  $f(\cdot,y) \in W^{1,p}(X)$ for $\nu$-almost every $y \in Y$, and 
\begin{equation}\label{eq:derivative}
	(x,y)\mapsto \sqrt{|Df(\cdot,y)|^2(x) + |Df(x,\cdot)|^2(y)}\in L^p(X\times Y).
\end{equation}
In addition, tensorization of Sobolev spaces requires that the minimal $p$-weak upper gradient of any $f\in J^{1,p}(X,Y)$ is given by \eqref{eq:derivative}. For the definition of $W^{1,p}(X)$ used in this paper, see section \ref{subsec:notations}.

While immediate in Euclidean spaces, a positive answer to the tensorization problem is non-trivial in the non-smooth setting, and needed e.g. in the splitting theorem for RCD-spaces \cite{gig18}. Further, it is of crucial importance in a  variety of settings where partial derivatives can be bounded, and one wishes to optain a bound on the full derivative, see e.g. \cite{AmbrosioPinamontiSpeight, ebgkns22}.  Surprisingly, the problem has remained open, even though tensorization of many other properties such as the doubling property, Poincar\'e inequalities and curvature lower bounds are well known. Previous partial results for $p=2$ include the work of Ambrosio--Gigli--Savar\'e \cite{ags14a} for RCD-spaces, of Gigli--Han \cite{giglihanwarped} settling the case where one factor is a closed interval in $\R$, and of Ambrosio--Pinamonti--Speight \cite{AmbrosioPinamontiSpeight} for PI-spaces. Working in the general case $p\ge 1$ (with a finite dimensionality assumption on the factors) the authors of the present manuscript proved tensorization of Sobolev spaces assuming \emph{one} of the factors is a PI-space \cite{teritapioseb}. The present work strengthens all of these results in the $p=2$ case, and proves stronger results for all $p>1$.


\subsection{Tensorization in infinitesimally quasi-Hilbertian spaces}

In this paper we establish tensorization of Sobolev spaces in the important special case $p=2$ when the factors are \emph{infinitesimally quasi-Hilbertian}. 

\begin{defn}\label{def:inf-quasi-hilb}
A metric measure space $X$ is infinitesimally quasi-Hilbertian if there exists a closed Dirichlet form $\mathcal E$ with domain $W^{1,2}(X)$ such that $\sqrt{\|u\|_{L^2(X)}^2+\mathcal E(u,u)}$ is an equivalent norm on $W^{1,2}(X)$.	
\end{defn}
See Section \ref{sec:quasi-hilb} for the definition of Dirichlet forms. We remark that infinitesimally Hilbertian spaces, as well as spaces admitting a $2$-weak differentiable structure (in particular spaces with finite Hausdorff dimension) are infinitesimally quasi-Hilbertian, cf. Proposition \ref{prop:dirichletexistence}.

\begin{thm}\label{thm:main}
	Suppose that $X$ and $Y$ are infinitesimally quasi-Hilbertian. Then \newline $W^{1,2}(X \times Y) = J^{1,2}(X,Y)$ and, for each $f\in J^{1,2}(X,Y)$, we have that 
	\begin{align*}
	|Df|(x,y)^2=|Df(\cdot,y)|(x)^2+|Df(x,\cdot)|(y)^2\quad 
	\end{align*}
for $\mu\times \nu$-almost every $(x,y)\in X\times Y$.
\end{thm}

\begin{remark}
If the space $X\times Y$ is equipped with a product metric $\|(d_X,d_Y)\|$ induced by some norm $\|\cdot\|$ on $\R^2$, we obtain that $|Df| = \|(|Df(\cdot,y)|(x), |Df(x,\cdot)|(y))\|'$ where  $\|\cdot\|'$ is a form of dual norm, cf. Theorem \ref{thm:gen_main}.
\end{remark}

In particular we have the following corollary.

\begin{cor}\label{cor:fin-dim}
If each of the factors $X$ and $Y$ is either infinitesimally Hilbertian or has finite Hausdorff dimension, then $W^{1,2}(X\times Y)=J^{1,2}(X,Y)$ with equal norms.
\end{cor}

Theorem \ref{thm:main} follows by combining three ingredients: 1) standard theory of Dirichlet forms and the elementary inclusion $W^{1,2}(X\times Y)\subset J^{1,2}(X,Y)$, 2) the non-trivial fact that the inclusion $W^{1,2}(X\times Y)\subset J^{1,2}(X,Y)$ has norm one, and 3) the equality of the norms on the algebraic tensor product $W^{1,2}(X)\otimes W^{1,2}(Y)$. We establish the last two results in the more general setting when $p> 1$ and the product space $X\times Y$ is equipped with a product metric given by a possibly non-Euclidean planar norm (see also \cite{teritapioseb} where the same setting is used). We will also address the applicability of the results to the case of $p=1$.

\subsection{Norm inequalities and equalities in the inclusion $W^{1,p}(X\times Y)\subset J^{1,p}(X,Y)$}

Let $(X\times Y,d,\mu\times\nu)$ be the product of two metric measure spaces $X=(X,d_X,\mu)$ and $Y=(Y,d_Y,\nu)$, where the product metric is given by $d=\|(d_X,d_Y)\|$ for a given planar norm $\|\cdot\|$, and let $p>1$.

For $f\in J^{1,p}(X,Y)$ we denote by $|D_Xf|$ and $|D_Yf|$ the $L^p(X\times Y)$-functions $(x,y)\mapsto |Df(\cdot,y)|(x)$ and $(x,y)\mapsto |Df(x,\cdot)|(y)$, respectively, and replace \eqref{eq:derivative} with the comparable quantity
\begin{align}\label{eq:j1p-ug}
\|(|D_Xf|,|D_Yf|)\|'\in L^p(X\times Y).
\end{align}
Here $\|(a,b)\|'\defeq \sup\{ at+bs:\ s,t\ge 0  \}$ is the \emph{partial dual norm} of $\|\cdot\|$. Notice that the Euclidean norm is its own partial dual and thus, for $d=\sqrt{d_X^2+d_Y^2}$, \eqref{eq:derivative} and \eqref{eq:j1p-ug} coincide. The first of the two results states that the minimal $p$-weak upper gradient always dominates \eqref{eq:j1p-ug}.
\begin{thm}\label{thm:embedding-norm-one}
	Let $p \in (1,\infty)$. If $f\in W^{1,p}(X\times Y)$ then $f\in J^{1,p}(X,Y)$ and
\begin{align}\label{eq:embedding-norm-one}
	\|(|D_Xf|,|D_Yf|)\|'\le |Df|
\end{align}
$\mu\times\nu$-almost everywhere.
\end{thm}
In particular, for the Euclidean product metric $d=\sqrt{d_X^2+d_Y^2}$ Theorem \ref{thm:embedding-norm-one} yields the inequality
\begin{align*}
	\sqrt{|D_Xf|^2+|D_Yf|^2}\le |Df|,\quad f\in W^{1,p}(X\times Y).
\end{align*}

Although it is straightforward to obtain the estimate $\|(|D_Xf|,|D_Yf|)\|'\le C|Df|$ for some $C>0$ independent of $f$ from the definitions, Theorem \ref{thm:embedding-norm-one} is new and was previously only known for general spaces when $p=2$ and $\|\cdot\|$ is the Euclidean norm, by work of  Ambrosio--Gigli--Savar\'e \cite{ags14b} via different techniques (see also \cite{AmbrosioPinamontiSpeight}).  Our approach uses a density in energy argument \cite{seb2020,ambgigsav} to reduce the proof of Theorem \ref{thm:embedding-norm-one} to a simple, yet novel, inequality for Lipschitz functions (see Proposition \ref{prop:lip_a-est} below).


Our next result establishes \emph{equality} in \eqref{eq:embedding-norm-one} in the \emph{algebraic tensor product} $W^{1,p}(X)\otimes W^{1,p}(Y)\subset W^{1,p}(X\times Y)$ consisting of finite sums of simple tensor products, i.e. functions of the form
\[
f(x,y)=\sum_j^N\varphi_j(x)\psi_j(y),\quad \varphi_j\in W^{1,p}(X),\ \psi_j\in W^{1,p}(Y),\quad j=1,\ldots,N.
\]
\begin{thm}\label{thm:norm-equality-on-tensor-prod}
	Let $p \in (1,\infty)$. If $f\in W^{1,p}(X)\otimes W^{1,p}(Y)$ then 
	\begin{align*}
		\|(|D_Xf|,|D_Yf|)\|'= |Df|
	\end{align*}
	$\mu\times\nu$-almost everywhere.
\end{thm}

We show, using \emph{canonical minimal upper gradients} introduced in \cite{teriseb}, that \eqref{eq:j1p-ug} is a $p$-weak upper gradient of $f\in W^{1,p}(X)\otimes W^{1,p}(Y)$. Together with Theorem \ref{thm:embedding-norm-one} this suffices to demonstrate Theorem \ref{thm:norm-equality-on-tensor-prod}. Note that having constant one in \eqref{eq:embedding-norm-one} is important for the validity of this argument.

\bigskip\noindent The crux of Theorem \ref{thm:main} is that, for infinitesimally quasi-Hilbertian spaces, the algebraic tensor product is dense in the Beppo--Levi space (with $p=2$). Indeed, this is a standard result for domains of Dirichlet forms (see Proposition \ref{prop:dirichlet-tensor-prod-density}), and follows easily for Sobolev spaces under the infinitesimal quasi-Hilbertianity assumption. This completes the proof of Theorem \ref{thm:main} and also raises the natural question: when is $W^{1,p}(X)\otimes W^{1,p}(Y)$ dense in $J^{1,p}(X,Y)$? We expect that some separability assumption might be necessary, and formulate the question accordingly below.

\begin{question}\label{q}
Let $p\in [1,\infty)$. If $W^{1,p}(X)$ and $W^{1,p}(Y)$ are separable, is $W^{1,p}(X)\otimes W^{1,p}(Y)$ dense in $J^{1,p}(X,Y)$?
\end{question}
An affirmative answer to Question \ref{q} under the stronger assumption that $X$ and $Y$ admit $p$-weak differentiable structures would already be interesting, since it covers all spaces with finite Hausdorff dimension.

\begin{remark}
 The above Theorems \ref{thm:embedding-norm-one} and \ref{thm:norm-equality-on-tensor-prod} are stated for exponents $p>1$. In the proofs we use the equality of the 
 Newton-Sobolev space $N^{1,p}(X)$  defined by Shanmugalingam and Cheeger \cite{sha00, che99}, and the plan-Sobolev space $W^{1,p}(X)$ from Ambrosio, Gigli and Savar\'e \cite{ags14a}. The equality of these spaces is not yet available in the literature in the case $p=1$. Once proven, such equality would imply Theorems \ref{thm:embedding-norm-one} and \ref{thm:norm-equality-on-tensor-prod} also in the case $p=1$.
\end{remark}


\subsection{Notation and conventions} \label{subsec:notations}
Throughout the paper $X=(X,d_X,\mu)$ and $Y=(Y,d_Y,\nu)$ are metric measure spaces, by which we mean complete separable metric spaces equipped with measures that are finite on bounded sets. Given $p> 1$, we denote by $N^{1,p}(X)$ and $W^{1,p}(X)$ the \emph{Newton--Sobolev space}, and Sobolev space via test-plans, respectively. Both of these spaces are defined using the \emph{upper gradient inequality}. A function $f\in L^p(X)$  is in $N^{1,p}(X)$ if there exists a function $g\in L^p(X)$ so that
\begin{align}\label{eq:ug1}
|u(\gamma_1)-u(\gamma_0)|\le \int_0^1 g(\gamma_t)|\gamma_t'|\ud t 
\end{align}
holds for $\Mod_p$-a.e. curve. On the other, $f\in W^{1,p}(X)$, if  \eqref{eq:ug1} holds for $\bm\eta$-a.e. $\gamma$ for every \emph{$q$-test plan} $\bm\eta$. Modulus is an outer measure on curve families, and test plans are a family of measures on curve families. See \cite{ags14a} for a definition of test plans. For the properties the modulus of a curve family, $\Mod_p$, see \cite{HKST07}.

For each $f\in N^{1,p}(X)$ and $f\in W^{1,p}(X)$ there exists a minimal $|Du|\in L^p(X)$ so that
\begin{align}\label{eq:ug}
|u(\gamma_1)-u(\gamma_0)|\le \int_0^1 |Du|(\gamma_t)|\gamma_t'|\ud t 
\end{align}
holds for ``almost all'' absolutely continuous curves $\gamma\colon[0,1]\to X$. For $u\in N^{1,p}(X)$, \eqref{eq:ug} is required to hold for $\Mod_p$-almost every curve $\gamma$ whereas, and for $u\in W^{1,p}(X)$, \eqref{eq:ug} holds for $\bm\eta$-a.e. $\gamma$ for every \emph{$q$-test plan} $\bm\eta$. The minimal objects $|Du|$ associated to each case agree $\mu$-almost everywhere (in this notation we suppress its dependence on $p$ and on the metric) and we have that $W^{1,p}(X)=N^{1,p}(X)$ for $p>1$ with equal norms\footnote{The equality holds up to the subtle issue of choosing appropriate representatives: $N^{1,p}(X)\subset W^{1,p}(X)$, but for every $f\in N^{1,p}(X)$ there exists a function $\tilde{f} \in W^{1,p}(X)$ with $\tilde{f}=f$ almost everywhere. A further difference is that functions in $N^{1,p}(X)$ are defined up to capacity-a.e. equivalence, whereas functions in $W^{1,p}(X)$ are defined up to an almost everywhere equivalence. The proof is contained in \cite[Theorem 10.7]{amb13}.} \cite{amb13,ambgigsav}. Here the Sobolev space $W^{1,p}(X)$ is equipped with norm
\begin{align*}
	\|u\|_{W^{1,p}(X)}=\big(\|u\|_{L^p(X)}^p+\||Du|\|_{L^p(X)}^p\big)^{1/p}.
\end{align*}

For functions $u\colon X\times Y\to\R$ we will define the sliced functions, for $x\in X, y\in Y$, by 
\[
u_x\defeq u(x,\cdot):Y\to \R,\quad u^y\defeq u(\cdot,y):X\to \R.
\]
When $u\in J^{1,p}(X,Y)$ we denote by $|D_Xu|,|D_Yu|\in L^p(X\times Y)$ the functions such that
\[
|D_Xu|(\cdot,y)=|Du^y|\quad \textrm{for $\nu$-a.e. $y\in Y$,}\quad  |D_Yu|(x\cdot,)=|Du_x|\quad  \textrm{for $\mu$-a.e. $x\in X$.}
\]
We equip $J^{1,p}(X,Y)$ with the norm
\begin{align*}
	\|f\|_{J^{1,p}(X,Y)}=\left(\int_{X\times Y}\big(|f|^p+(\|(|D_Xf|,|D_Yf|)\|')^p\big)\ud(\mu\times\nu)\right)^{1/p}.
\end{align*}

\begin{remark}\label{rmk:minimal_directional_ug}
It is straightforward to check that $|D_Xu|$ is given as the minimal $p$-weak upper gradient of $u$ when $X\times Y$ is equipped with the metric $d=d_X+\sqrt d_Y$, and similarly for $|D_Yu|$. In particular $|D_Xu|$ and $|D_Yu|$ can be chosen Borel measurable. 
\end{remark}


\section{The inclusion $W^{1,p}(X\times Y)\subset J^{1,p}(X,Y)$}

In this section we prove Theorems \ref{thm:embedding-norm-one} and \ref{thm:norm-equality-on-tensor-prod} for a general $p>1$. 
In the proof of Proposition \ref{prop:lip_a-est} below we will shorten the notation by using the evaluation 
map
\[
e_X\colon C([0,1];X)\times [0,1]\to X \colon (\gamma,t)\mapsto \gamma_t.
\]

%

We start with a seemingly elementary inequality, which however has hitherto not appeared. The authors in \cite{AmbrosioPinamontiSpeight} and \cite{ags14b} used a substantially different approach employing Hopf-lax equations, heat flows and further results. The following result is the key to our proof of the isometric inclusion, and perhaps gives a more transparent and geometric argument. 

\begin{prop}\label{prop:lip_a-est}
Let $p \in (1,\infty)$ and $f\in \LIP_b(X\times Y)$. Then 
\begin{align*}
\|(|D_Xf|,|D_Yf|)\|'\le \Lip_af
\end{align*}
$\mu\times\nu$-almost everywhere.
\end{prop}

\begin{proof}[Proof of Proposition \ref{prop:lip_a-est}]
%
The argument will proceed by finding, for a.e. point $(x,y)$ a curve in the $X$- and $Y$-directions along which the function $f$ has maximal derivative given by the minimal $p$-weak upper gradients. We do this by employing a result from \cite{teriseb}, but we also outline in Remark \ref{rmk:furtherargument} another argument inspired by one from  Cheeger and Kleiner \cite{cheegerkleiner} after the proof, which some readers may find helpful.

 Since $f \in \LIP_b(X\times Y)$, we have that $f \in W^{1,p}(X \times Y)$ when $X \times Y$ is equipped with the distance $d_X + \sqrt{d_Y}$. By Remark \ref{rmk:minimal_directional_ug} and \cite[Theorem 1.1]{teriseb} there exists a test plan $\bm\eta$ so that the disintegration $\{\bm\pi_{(x,y)}\}$ of the measure $\ud\bm\pi\defeq |\gamma'_t|\ud t\ud\bm\eta$ with respect to the evaluation map $e_{X\times Y}:C([0,1];X\times Y)\times [0,1]\to X\times Y$ satisfies

\begin{equation}\label{x-derivative}
    |D_Xf|(x,y)=|Df^y|(x)=\left\|\frac{(f^y\circ\alpha)'_t}{|\alpha_t'|}\right\|_{L^\infty(\bm\pi_{(x,y)})} 
\end{equation}
for $\mu\times\nu$-almost every $(x,y)\in \{|D_Xf|>0\}$. (Notice that every rectifiable curve in $(X\times Y,d_X+\sqrt{d_Y})$ is of the form $(\alpha,y)$ where $y\in Y$ is a constant curve and $\alpha$ is a rectifiable curve in $X$. One could obtain \eqref{x-derivative} for $p>1$ alternatively via the existence of master test plans introduced in \cite{pasqualetto20b} and by using Fubini's theorem.) By applying the same argument with metric $\sqrt{d_X}+d_Y$ we similarly obtain measures $\{\widetilde{\bm\pi}_{(x,y)}\}$ for almost every $(x,y) \in\{|D_Yf|>0\}$ so that 
\begin{equation}\label{y-derivative}
    |D_Yf|(x,y)=|Df_x|(y)=\left\|\frac{(f_x\circ\alpha)'_t}{|\alpha_t'|}\right\|_{L^\infty(\widetilde{\bm\pi}_{(x,y)})}.
\end{equation}

Let us fix $(x,y) \in X \times Y$ where both \eqref{x-derivative} and \eqref{y-derivative} hold.
For any $\eps>0 $ there exist $(\alpha,t_0)\in e_X\inv(x)$ and $(\beta,s_0)\in e_Y\inv(y)$ such that
\begin{align}\label{eq:almost_maximal_curves}
(1-\varepsilon)|D_Xf|(x,y)\le \frac{(f^y\circ\alpha)'_{t_0}}{|\alpha_{t_0}'|},\quad 
(1-\varepsilon)|D_Yf|(x,y)\le \frac{(f_x\circ\beta)'_{s_0}}{|\beta_{s_0}'|},
\end{align}
and the limits in all the relevant quantities exist. Let $a,b\ge 0$ and define the curves 
\begin{align*}
\tilde\alpha(t)=\alpha\left(t_0+\frac{a}{|\alpha'_{t_0}|}t\right),\quad \tilde \beta(s)=\beta\left(s_0-\frac{b}{|\beta'_{s_0}|}s\right)
\end{align*}
in a small neighbourhood of the origin. Then 
\begin{align*}
&a\frac{(f^y\circ\alpha)'_{t_0}}{|\alpha_{t_0}'|}+b\frac{(f_x\circ\beta)'_{s_0}}{|\beta_{s_0}'|}=(f^y\circ\tilde\alpha)'_0-(f_x\circ\tilde\beta)'_0\nonumber \\
& =\lim_{h\to 0^+}\frac{[f(\tilde\alpha(h),y)-f(x,y)]-[f(x,\tilde\beta(h))-f(x,y)]}{h}=\lim_{h\to 0^+}\frac{f(\tilde\alpha(h),y)-f(x,\tilde\beta(h))}{h}\nonumber\\
& \le \Lip_af(x,y)\limsup_{h\to 0^+}\frac{d(\tilde\alpha(h),y),(x,\tilde\beta(h)))}{h}.
\end{align*}
Note however that
\begin{align*}
	\frac{d(\tilde\alpha(h),y),(x,\tilde\beta(h)))}{h}=\left\|\left(\frac{d_X(\tilde\alpha(h),x)}{h},\frac{d_Y(\tilde\beta(h),y)}{h} \right)\right\| \stackrel{h\to 0^+}{\longrightarrow} \|(a,b)\|.
\end{align*}
Using \eqref{eq:almost_maximal_curves} we arrive at
\begin{equation}\label{eq:crucial_estimate}
(1-\eps)[a|D_Xf|(x,y)+b|D_Yf|(x,y)]\le \|(a,b)\|\Lip_af(x,y).
\end{equation}
Taking supremum over all $a,b\ge 0$ with $\|(a,b)\|=1$ in \eqref{eq:crucial_estimate} yields
\begin{align*}
	(1-\eps)\|(|D_Xf|(x,y),|D_Yf|(x,y))\|'\le \Lip_af(x,y)\quad \mu\times\nu\text{-a.e. } (x,y)\in X\times Y.
\end{align*}
Since $\eps>0$ is arbitrary the claim now follows.
\end{proof}

\begin{remark}\label{rmk:furtherargument} In the previous proof, the use of \cite{teriseb} is convenient, but the existence of curves $\alpha$ and $\beta$ as in \eqref{eq:almost_maximal_curves} with nearly maximal derivative is actually a much weaker conclusion. In fact, in the category of doubling spaces satisfying a Poincar\'e inequality, their existence follows from the work of Cheeger and Kleiner \cite[Theorem 5.2]{cheegerkleiner}. They gave a characterization of the minimal $p$-weak upper gradient of a Lipschitz function as a maximal directional derivative. In fact, the first part of the proof, which does not use the doubling or Poincar\'e assumptions, shows that a function $\hat{g}$ defined using the maximal directional derivatives, is an upper gradient. A minimal p-weak upper gradient is a.e. less than this upper gradient, and from this the existence of $\alpha$ and $\beta$ can be deduced. This idea played a central role in later developments, such as the seminal work of Bate \cite{bate12diff} characterizing Lipschitz differentiability spaces. 
\end{remark}

Proposition \ref{prop:lip_a-est} now implies Theorem \ref{thm:embedding-norm-one}.

\begin{proof}[Proof of Theorem \ref{thm:embedding-norm-one}]
Let $f\in W^{1,p}(X)$. By the density in energy (cf. \cite{ambgigsav} or \cite{seb2020} for an alternate proof) there exists a sequence $(f_j)\subset \LIP_b(X)$ such that $f_j\to f$ and $\Lip_af_j\to |Df|$ in $L^p(X)$ as $j\to \infty$. We also have, for $a,b\ge 0$ and any non-negative $\varphi\in C_b(X)$, that 
\begin{align*}
\int_{X\times Y}\varphi (a|D_Xf|+b|D_Yf|)\ud(\mu\times\nu)\le \liminf_{j\to\infty}\int_{X\times Y}\varphi(a|D_Xf_j|+b|D_Yf_j|)\ud(\mu\times\nu),
\end{align*}
(cf. Remark \ref{rmk:minimal_directional_ug} and the lower semicontinuity of the Cheeger energy \cite{che99,ambgigsav}). By Proposition \ref{prop:lip_a-est} (and the definition of the partial dual norm $\|\cdot\|'$) this implies that
\begin{align*}
 \int_{X\times Y}\varphi (a|D_Xf|+b|D_Yf|)\ud(\mu\times\nu) \le & \|(a,b)\|\liminf_{j\to\infty}\int_{X\times Y}\varphi \Lip_af_j\ud(\mu\times\nu)\\
= & \|(a,b)\|\int_{X\times Y}\varphi|Df|\ud(\mu\times\nu)
\end{align*}
for arbitrary $a,b$ and $\varphi$. Thus $a|D_Xf|+b|D_Yf|\le \|(a,b)\||Df|$ $\mu\times\nu$-a.e. for every $a,b\ge 0$. Then, by choosing a countable dense set of real numbers $a,b\geq 0$ we obtain the pointwise inequality
\[
\|(|D_Xf|,|D_Yf|)\|'=\sup_{(a,b)}\frac{a|D_Xf|+b|D_Yf|}{\|(a,b)\|}\le |Df|\quad\mu\times \nu-a.e.
\]

\end{proof}

Next we prove Theorem \ref{thm:norm-equality-on-tensor-prod}. In the proof we identify $\R^N$ with $(\R^N)^*$ in the standard way by identifying $\bar a\in \R^N$ with the functional $x\mapsto \bar a\cdot x\in (\R^N)^*$.

\begin{proof}[Proof of Theorem \ref{thm:norm-equality-on-tensor-prod}]
Let 
\[
h(x,y)=\sum_{i=1}^N f_i(x)g_i(y) \in W^{1,p}(X)\otimes W^{1,p}(Y)
\]
where $f_1, \dots, f_N \in W^{1,p}(X)$ and $g_1,\dots, g_N \in W^{1,p}(Y)$. Since each function in $W^{1,p}(X)$ is a.e. equal to a Newton-Sobolev function \cite[Theorem 10.7]{amb13}, we can choose Newton-Sobolev representatives for each $f_j$ and $g_j$ and consider the maps $\varphi=(f_1,\dots, f_N) \in N^{1,p}(X;\R^N)$ and $\psi =(g_1,\dots, g_N) \in N^{1,p}(Y;\R^N)$. By \cite[Lemma 4.2]{teriseb}, we have the following:
\begin{enumerate}
\item there exist maps $\Phi:X\times (\R^N)^* \to [0,\infty]$ and $\Psi:Y\times (\R^N)^*\to [0,\infty]$ so that $\Phi(\cdot, \bm \xi)$ is the minimal $p$-weak upper gradient for $\bm \xi \circ \varphi$ and $\Psi(\cdot, \bm\zeta)$ is the minimal $p$-weak upper gradient of $\bm\zeta \circ \psi$ for every $\bm\xi,\bm\zeta\in (\R^N)^*$; 
\item there are families of curves $\Gamma_X,\Gamma_Y$ with $\Mod_p(\Gamma_X)=\Mod_p(\Gamma_Y)=0$ so that for every $\alpha\not\in \Gamma_X$ and every $\bm\xi\in (\R^N)^*$ the function $\bm\xi\circ\varphi$ is absolutely continuous on $\alpha$ with upper gradient $\Phi(x,\bm\xi)$, and for every  $\beta\not\in \Gamma_Y$ and every $\bm\zeta\in (\R^N)^*$ the function $\bm\zeta\circ\psi$ is absolutely continuous on $\beta$ with upper gradient $\Psi(x,\bm\zeta)$; and
\item for each $\alpha\not\in \Gamma_X$ and each $\beta\not\in \Gamma_Y$ there exist null sets $E_\alpha \subset [0,1]$, $E_\beta\subset [0,1]$  so that for every $\bm\xi,\bm\zeta\in (\R^N)^*$ we have
\[ |(\bm\xi \circ \varphi\circ \alpha)_t'| \leq \Phi(\alpha_t,\bm\xi) |\alpha_t'|, \text{ for every } t\in [0,1] \setminus E_\alpha\]
and
\[ |(\bm\zeta \circ \psi \circ \beta)_t'| \leq \Psi(\beta_t,\bm\zeta) |\beta_t'|, \text{ for every } t\in [0,1] \setminus E_\beta.\]
\end{enumerate}

First, we show that $h\in W^{1,p}(X\times Y)$ and  that $|Dh| \leq \|(|D_X h|, |D_Y h|)\|'$, which follows from showing that $g\defeq \|(|D_X h|, |D_Y h|)\|'$ is a $p$-weak upper gradient of $h$.

Note that $|D_X h| (x,y)= \Phi(x,(g_i(y)))$  and $|D_Y h|(x,y) = \Psi(y,((f_i(x)))$ for $\mu\times\nu$-almost every $x,y$. Thus, it suffices to show that $\|(\Phi(x,(g_i(y))), \Psi(y,(f_i(x)))\|$ is a $p$-weak upper gradient of $h$. Let $\Gamma$ be the collection of absolutely continuous curves  $\gamma=(\alpha,\beta)$ so that $\alpha\not\in \Gamma_X,\beta\not\in \Gamma_Y$. The complement of $\Gamma$ has zero $\Mod_p$-modulus, as follows fairly directly from  the definition of modulus and the characterization of families of zero modulus  \cite[Lemma 5.2.8]{HKST07}: there exists a function $g\in L^p(X\times Y)$ so that $\int_\gamma g \, \ud s = \infty$ for each $\gamma \not\in \Gamma$). 

Fix $\gamma\in \Gamma$. Since $f_i \circ \alpha$ and $g_i \circ \beta$ are absolutely continuous, so is $h$ as a product and sum of absolutely continuous functions. Further, it is differentiable a.e. and the Leibniz rule applies:

\[
(h\circ \gamma)'_t =  \sum_{i=1}^N (f_i\circ \alpha)'_t g_i(\beta(t))+ \sum_{i=1}^N f_i(\alpha(t)) (g_i \circ \beta)_t'.
\]

Now, 

\begin{align*}
\left| \sum_{i=1}^N (f_i\circ \alpha)'_t g_i(\beta(t))\right|\leq \Phi(\alpha(t), (g_i(\beta(t)))|\alpha_t'|, \quad t\not\in E_\alpha,\textrm{ and} \\
\left| \sum_{i=1}^N (g_i\circ \beta)'_t f_i(\alpha(t))\right|\leq \Psi(\beta(t), (f_i(\beta(t)))|\beta_t'|,\quad t\not\in E_\beta.
\end{align*}
Thus, for a.e. $t$, we have
\begin{align*}
|(h\circ\gamma)_t'| &\leq \Phi(\alpha(t),(g_i(\beta(t))))|\alpha_t'|+\Psi(\beta(t), (f_i(\beta(t)))|\beta_t'|\\
&\leq \|(\Phi(\alpha(t),(g_i(\beta(t)))),\Psi(\beta(t),(f_i(\alpha(t)))))\|'\|(|\alpha_t'|,|\beta_t'|)\|\\
&= \|(\Phi(\alpha(t),(g_i(\beta(t)))),\Psi(\beta(t),(f_i(\alpha(t)))))\|'|\gamma_t'| )= g(\gamma_t) |\gamma_t'|.
\end{align*}
By integrating this, we obtain the upper gradient inequality \eqref{eq:ug1}. This shows that $g$ is a $p$-weak upper gradient of $h$, whence $|Dh| \leq \|(|D_X h|, |D_Y h|)\|'$ holds $\mu\times\nu$-almost everywhere. Theorem \ref{thm:embedding-norm-one} gives the opposite inequality, and completes the proof of the claim.
\end{proof}

\section{Infinitesimally quasi-Hilbertian spaces}\label{sec:quasi-hilb}

In this section we complete the proof of Theorem \ref{thm:main}. We begin with the definition of closed Dirichlet forms. 

\begin{defn}\label{def:dirichletform}
Let $\mathcal{D} \subset L^2(X)$ be a vector sub-space. A Dirichlet form $\cE$ (with domain $\cD$) is a map $\cE\colon \cD\times \cD \to \R$, which satisfies:
\begin{enumerate}
\item $\cE$ is bilinear,
\item $\cE$ is symmetric, i.e. $\cE(u,v)=\cE(v,u)$ for each $u,v\in \cD$,
\item $\cE$ is non-negative, i.e. $\cE(u,u) \geq 0$ for each $u\in \cD$, and
\item $\cD$ is dense in $L^2(X)$.
\end{enumerate}
We say that $\cE$ is closed, if $\cD$ when equipped with the norm $\|f\|_{\cE} \defeq \sqrt{\|f\|_{L^2(X)}^2+\cE(f,f)}$ is complete. 
\end{defn}

Next we recall the tensor product of Dirichlet forms. Recall the notation $u^y=u(\cdot,y)$ and $u_x=u(x,\cdot)$ for $u\colon X\times Y\to\R$ and $(x,y)\in X\times Y$. If $(\mathcal E_X,\mathcal D_X)$ and $(\mathcal E_Y,\mathcal D_Y)$ are Dirichlet forms on $X$ and $Y$, respectively, the domain of their tensor product $(\mathcal E,\mathcal D)$ is defined by
\begin{align*}
\mathcal D=\bigg\{ u\in L^2(X\times Y):\ u^y\in \mathcal D_X\ \nu\textrm{-a.e. }y\in Y,\ u_x\in \mathcal D_Y\ \mu\textrm{-a.e. }x\in X,\ \mathcal E(u,u)<\infty \bigg\},
\end{align*}
where
\begin{align*}
	\mathcal E(u,u)\defeq \int_X \mathcal E_Y(u_x,u_x)\ud\mu(x)+\int_Y\mathcal E_Y(u^y,u^y)\ud\nu(y).
\end{align*}
The bilinear form $\mathcal E$ is given by polarization:
\[
\mathcal E(u,v)=\frac{\mathcal E(u+v,u+v)-\mathcal E(u-v,u-v)}{4}.
\]
We observe that the algebraic tensor product $\mathcal D_X\otimes \mathcal D_Y\subset L^2(X\times Y)$ is contained in $\mathcal D$. Here the algebraic tensor product is given by 
\[
\mathcal D_X \otimes \mathcal D_Y = \left\{\sum_{i=1}^N a_i(x)b_i(y) : a_i\in \mathcal D_X, b_i \in \mathcal D_Y, N\in \N\right\}.
\]
 We refer to \cite[Chapter V]{bou-hir91} for the basic properties of the tensor product of Dirichlet forms, and record here the density of $\mathcal D_X\otimes \mathcal D_Y$ in $\mathcal D$, cf. \cite[Proposition 2.1.3(b)]{bou-hir91}.

\begin{prop}\label{prop:dirichlet-tensor-prod-density}
Let $(\mathcal E_X,\mathcal D_X)$ and $(\mathcal E_Y,\mathcal D_Y)$ be closed Dirichlet forms on $X$ and $Y$, and let $\mathcal E$ be their tensor product. Then the algebraic tensor product $\mathcal D_X\otimes \mathcal D_Y$ is dense in $\mathcal D$ with respect to $\|\cdot\|_\mathcal E$.
\end{prop}

We will next see how Proposition \ref{prop:dirichlet-tensor-prod-density} implies the density of $W^{1,2}(X)\otimes W^{1,2}(Y)$ in $J^{1,2}(X,Y)$ (and thus Theorem \ref{thm:main}) for infinitesimally quasi-Hilbertian spaces. Recall that $X$ is said to be infinitesimally quasi-Hilbertian, if there exists a Dirichlet form $(\mathcal E,\mathcal D)$ on $X$ with $\mathcal D=W^{1,2}(X)$ and $\|\cdot\|_\mathcal E$ equivalent to $\|\cdot\|_{W^{1,2}(X)}$. Theorem \ref{thm:main} is a special case of the following theorem with $\|(a,b)\|\defeq \sqrt{a^2+b^2}$.

\begin{thm}\label{thm:gen_main}
	Suppose $X$ and $Y$ are infinitesimally quasi-Hilbertian and equip the product space $X\times Y$ with the metric $d=\|(d_X,d_Y)\|$ for a given norm $\|\cdot\|$ on $\R^2$. Then $W^{1,2}(X\times Y)=J^{1,2}(X,Y)$ and 
	\begin{align*}
		\|(|D_Xf|,|D_Yf|)\|'=|Df|
	\end{align*}
for every $f\in J^{1,2}(X,Y)$.
\end{thm}
In the proof we use the notation $A(u)\lesssim B(u)$ to indicate that there exists a constant $C>0$, independent of $u$ such that $ B(u)\le CA(u)$, and $A(u)\simeq B(u)$ if $A(u)\lesssim B(u)$ and $B(u)\lesssim A(u)$.
\begin{proof}
Let $\mathcal E_X$ and $\mathcal E_Y$ be closed Dirichlet forms on $X$ and $Y$ with domains $W^{1,2}(X)$ and $W^{1,2}(Y)$, respectively, such that $\|f\|_{\mathcal E_X}\simeq \|f\|_{W^{1,2}(X)}^2$ and $\|g\|_{\mathcal E_Y}\simeq \|g\|_{W^{1,2}(Y)}^2$ for all $f\in W^{1,2}(X)$ and $g\in W^{1,2}(Y)$. Then
\begin{align*}
	\|u\|_{J^{1,2}(X,Y)}^2& =\|u\|_{L^2(X\times Y)}^2+\int_{X\times Y}(\|(|D_Xu|,|D_Yu|)\|')^2\ud(\mu\times\nu)\\
	&\simeq \int_{X\times Y}(|u|^2+|D_Xu|^2+|D_Yu|^2)\ud(\mu\times\nu) \\
	 &\simeq  \int_X\|u_x\|_{W^{1,2}(Y)}^2\ud\mu(x)+\int_Y\|u^y\|_{W^{1,2}(Y)}^2\ud\nu(y)\\
	&\simeq \int_X(\|u_x\|_{L^2(Y)}^2+\mathcal E_Y(u_x,u_x))\ud\mu(x)+\int_Y\mathcal (\|u^y\|_{L^2(X)}^2+\mathcal E_X(u^y,u^y))\ud\nu(y) \\
	&\simeq  \|u\|_{L^2(X\times Y)}^2+\mathcal E(u,u)
\end{align*}
whenever $u\in L^2(X\times Y)$ is such that $u_x\in W^{1,2}(Y)$ for $\mu$-a.e. $x\in X$ and $u^y\in W^{1,2}(X)$ for $\nu$-a.e. $y\in Y$. From this it follows that $\mathcal D=J^{1,2}(X,Y)$ and that 
\begin{align}\label{eq:dir-norm-equiv-j12}
	\|u\|_\mathcal E\simeq\|u\|_{J^{1,2}(X,Y)},\quad u\in J^{1,2}(X,Y).
\end{align}

We now prove the claim in Theorem \ref{thm:gen_main}. Let $u\in J^{1,2}(X,Y)$. By Proposition \ref{prop:dirichlet-tensor-prod-density} there is a sequence $(u_j)\subset W^{1,2}(X)\otimes W^{1,2}(Y)$ such that $\|u_j-u\|_\mathcal E\to 0$ as $j\to \infty$. Theorem \ref{thm:norm-equality-on-tensor-prod} implies that 
\begin{align*}
\|u_j-u_l\|_{W^{1,2}(X\times Y)}&=\|u_j-u_l\|_{J^{1,2}(X,Y)}\lesssim \|u_j-u_l\|_\mathcal E,\\
g_j&=|Du_j|
\end{align*}
for each $j,l\in\N$. Thus $(u_j)$ is a Cauchy sequence in $W^{1,2}(X\times Y)$ and its limit (in $W^{1,2}(X\times Y)$) agrees almost everywhere with $u$ since $u$ is the $L^2$-limit of $(u_j)$. It follows that $u\in W^{1,2}(X\times Y)$ and, by Theorem \ref{thm:embedding-norm-one}, that $\|(|D_Xu|,|D_Yu|)\|'\le |Du|$. However, since $u_j\to u$ in $J^{1,2}(X,Y)$ as $j\to\infty$, we have that
$$|Du_j|=\|(|D_Xu_j|,|D_Yu_j|)\|' \stackrel{j\to\infty}{\longrightarrow} \|(|D_Xu|,|D_Yu|)\|'\quad\textrm{in }L^2(X\times Y)$$
and thus $\|(|D_Xu|,|D_Yu|)\|'$ is a 2-weak upper gradient of $u$ (cf. \cite[Proposition 7.3.7]{HKST07}), implying that $|Du|\le \|(|D_Xu|,|D_Yu|)\|'$. This completes the proof.
\end{proof}

\begin{remark}
	The proof of Theorem \ref{thm:gen_main} also yields the following statement: If $X$ and $Y$ are infinitesimally quasi-Hilbertian then their product $X\times Y$ with the metric $d=\|(d_X,d_Y)\|$ is also infinitesimally quasi-Hilbertian.
\end{remark}

Infinitesimally Hilbertian spaces are infinitesimally quasi-Hilbertian (recall that $X$ is infinitesimally Hilbertian if $\|\cdot\|_{W^{1,2}(X)}$ is given by an inner product). Another important class of examples are spaces $X$ admitting a $2$-weak differentiable structure or, equivalently, spaces with finitely generated tangent module $L^2(T^*X)$ in the sense of Gigli. To prove this fact we employ the following lemma on norms on finite dimensional vector spaces; see \cite[p.~460]{che99} for an original reference.

\begin{lemma}\label{lem:equivnorm} For every $k\in \N$ there exists a constant $c(k)$ so that for every $|\cdot|$ norm on a $k$-dimensional vector space $V$, with $k\in \N$, there exists an inner product $\langle, \rangle$ on $V$ so that 
\[
 c(k)^{-1} \sqrt{\langle v,v\rangle} \leq |v|\leq c(k) \sqrt{\langle v,v\rangle}.
\]
\end{lemma}
\begin{proof}
Let $V^*$ be the dual vector space to $V$ equipped with the dual norm $\|v^*\|=\sup_{v\in V, |v|=1} \langle v^*,v \rangle$. Let $B_1 \subset V^*$ be the closed unit ball with respect to this norm, and $\lambda$ the $k$-dimensional Hausdorff measure associated to the natural metric on $V^*$. 
By a classic Lemma of Kirchheim, $\lambda(B_1)=\omega_k$, where $\omega_k$ is the volume of the $k$-dimensional unit ball; see \cite[Lemma 6]{kir94}.

 We define the inner product via the natural map $V\to(V^*)^* \to L^\infty(B_1) \to L^2(B_1,\lambda)$, and set
\[
\langle v,w\rangle = \int_{B_1} \langle v,v^* \rangle \overline{\langle w, v^* \rangle } d\lambda_{v^*}. 
\]
The required inequality is true for $v=0$, and thus we consider the case where $v\neq 0$.
We immediately get $\langle v,v\rangle \leq \|v\|^2 \omega_k$, since $\langle v,v^*\rangle \leq 1$, for each $v^* \in B_1$. 
Next, take a $w^* \in B_1$ so that  $\langle w^*,v\rangle = |v|$. Let $\omega = \frac{w^*}{2}$, for which we have $B=B_{1/4}(\omega) \subset B_1$. 
Further, for every $a^* \in B$ we have $\langle a^*,v\rangle \geq \langle \omega, v\rangle - \langle \omega-a^*,v\rangle \geq |v|/2-|v|/4\geq |v|/4$. Thus, 
\[
\langle v,v\rangle \geq \lambda(B) |v|^2/16 \geq \frac{\omega_k}{4^{k+2}}|v|^2,
\]
from which the claim follows.
\end{proof}

\begin{prop}\label{prop:dirichletexistence} Suppose $X$ is a metric measure space such that $L^2(T^*X)$ is finitely generated. Then there exists a closed (and local and regular) Dirichlet form $\mathcal E$ with domain $W^{1,2}(X)$ and such that for some $C>0$
\[
\frac 1C \|Du\|_{L^2(X)}^2\le \mathcal E(u,u)\le C\|Du\|_{L^2(X)}^2,\quad u\in W^{1,2}(X).
\]
In particular, $X$ is infinitesimally quasi-Hilbertian.
\end{prop}

\begin{proof}
By \cite[Theorem 1.11]{teriseb}, $X$ admits a $2$-weak differential structure, that is, there are countably many disjoint 2-weak charts $(U_i, \varphi_i)$, $i\in I$, such that $\mu(X\setminus \bigcup_iU_i)=0$ and the dimensions of the Lipschitz maps $\varphi_i\colon U_i \to \R^{n_i}$ satisfy $N\defeq \sup_in_i<\infty$. Moreover, for each $x\in U_i$, $i\in I$, there is a norm $|\cdot|_x$ on $(\R^{n_i})^*$ such that
\begin{align*}
|Df|(x)=|\ud_xf|_x\quad \mu-a.e.\ x\in U_i
\end{align*}
for every $f\in W^{1,2}(X)$, where $\ud f \colon U_i\to (\R^{n_i})^*$ is the $p$-weak differential of $f$.

By Lemma \ref{lem:equivnorm} and its proof, there exists $c(N)>$ so that for each $i\in I$ and $x\in U_i$ there exists an inner product $\langle\cdot,\cdot\rangle_x$ on $(\R^{n_i})^*$ with the property that $x\mapsto \langle \ud_xf,\ud_xf\rangle_x\colon U_i\to \R$ is measurable, and
\begin{align*}
c(N)\inv\sqrt{\langle\ud_xf,\ud_xf\rangle_x}\le |\ud_xf|_x\le c(N)	\sqrt{\langle\ud_xf,\ud_xf\rangle_x}
\end{align*}
for every $f\in W^{1,2}(X)$ and $x\in U_i$.

Define a bi-linear form $\mathcal E\colon W^{1,2}(X)\times W^{1,2}(X)\to \R$ by
\begin{align*}
	\mathcal E(u,v)=\sum_{i\in I}\int_{U_i}\langle\ud_xu,\ud_x v\rangle_x\ud\mu(x).
\end{align*}
Clearly $\mathcal E$ is a Dirichlet form with domain $W^{1,2}(X)$ and 
\begin{align*}
c(N)^{-2}\|Du\|_{L^2(X)}^2\le \mathcal E(u,u)\le c(N)^2\|Du\|_{L^2(X)}^2,\quad u\in W^{1,2}(X).
\end{align*}
Since $W^{1,2}(X)$ is a Banach space it follows that $\mathcal E$ is closed. It is not difficult to check that $\mathcal E$ is local and regular, and we leave it to the interested reader (see e.g. \cite[Chapter I]{bou-hir91}) This completes the proof. 
%
%
\end{proof}

\begin{proof}[Proof of Corollary \ref{cor:fin-dim}]
Infinitesimally Hilbertian spaces are trivially infinitesimally quasi-Hilbertian. Spaces of finite Hausdorff dimension admit a 2-weak differentiable structure and their $L^2$-cotangent module is finitely generated, cf. \cite[Theorems 1.5 and 1.11]{teriseb}. By Proposition \ref{prop:dirichletexistence} spaces of finite Hausdorff dimension are therefore infinitesimally quasi-Hilbertian. Corollary \ref{cor:fin-dim} follows immediately from this and Theorem \ref{thm:main}.
\end{proof}

\bibliographystyle{plain}
\bibliography{abib}
\end{document}